\documentclass[11pt, reqno]{amsart}
\usepackage{indentfirst, amssymb, amsmath, amsthm, mathrsfs, setspace, indentfirst, enumerate,  mathrsfs, amsmath, amsthm, MnSymbol}
\usepackage[bookmarksnumbered, colorlinks, plainpages]{hyperref}
\usepackage{mathrsfs}
\usepackage{tikz}
\usetikzlibrary{positioning,calc}
\usepackage{pgfplots}
\usepackage{tabularx, colortbl, xcolor}
\usepackage{array, enumitem}
\pgfplotsset{compat=1.18}
\usetikzlibrary{3d}
\usepackage{tikz-3dplot}
\tdplotsetmaincoords{70}{120}
\newtheorem*{theoA}{Theorem A}
\newtheorem*{theoB}{Theorem B}

\newtheorem*{exmA}{Example A}

\newtheorem*{cor A}{Corollary A}
\newtheorem*{cor B}{Corollary B}

\newtheorem{theo}{Theorem}[section]
\newtheorem{lem}{Lemma}[section]

\newtheorem{exm}{Example}[section]

\newtheorem{rem}{Remark}[section]
\newcommand{\ol}{\overline}
\newcommand{\be}{\begin{equation}}
\newcommand{\ee}{\end{equation}}
\newcommand{\beas}{\begin{eqnarray*}}
\newcommand{\eeas}{\end{eqnarray*}}
\newcommand{\bea}{\begin{eqnarray}}
\newcommand{\eea}{\end{eqnarray}}

\numberwithin{equation}{section}
\begin{document}
\title[R\MakeLowercase{igidity of Entire Functions Sharing a Finite Set with Their Partial Derivatives in} $\mathbb{C}^n$]{\LARGE R\LARGE\MakeLowercase{igidity of Entire Functions Sharing a Finite Set with Their Partial Derivatives in} $\mathbb{C}^n$}
\date{}
\author[S. M\MakeLowercase{ajumder}, A. B\MakeLowercase{anerjee and} S. P\MakeLowercase{anja}]{S\MakeLowercase{ujoy} M\MakeLowercase{ajumder}, A\MakeLowercase{bhijit} B\MakeLowercase{anerjee and} S\MakeLowercase{hantanu} P\MakeLowercase{anja$^*$}}
\address{Department of Mathematics, Raiganj University, Raiganj, West Bengal-733134, India.}
\email{sm05math@gmail.com, sjm@raiganjuniversity.ac.in}

\address{ Department of Mathematics, University of Kalyani, West Bengal 741235, India.}
\email{abanerjee\_kal@yahoo.co.in, abanerjeekal@gmail.com}
	
\address{Department of Mathematics, University of Kalyani, West Bengal 741235, India.}
\email{panjasantu07@gmail.com}

\renewcommand{\thefootnote}{}
\footnote{*\emph{Corresponding Author}: Shantanu Panja.}
\renewcommand{\thefootnote}{\arabic{footnote}}
\setcounter{footnote}{0}

		\begin{abstract}
			This paper investigates certain classes of entire functions in $\mathbb{C}^n$ that, together with their partial derivatives, share a finite set consisting of three elements. By employing normality criteria, we study the behaviour of such functions and derive the necessary conditions governing their existence. Our results extend those of \cite{Chang-Fang-Zalcman-2007}, originally established for functions of a single complex variable, to the setting of several complex variables, thereby providing a comprehensive generalization of the earlier result in a direction not previously explored.
		\end{abstract}
	
\thanks{Typeset by \AmS -\LaTeX}
\maketitle

\section{{\bf Introduction and main result}}

The theory of value sharing between a meromorphic function and its derivatives occupies a central position in modern complex analysis, particularly within the framework of uniqueness theory. Over the past decades, a considerable body of research has been devoted to understanding how the distribution of shared values constrains the analytic structure of entire and meromorphic functions. A remarkable direction in this area concerns identifying minimal sets that enforce strong rigidity, ultimately forcing a function to coincide with its derivative.

In this context, the pioneering work of Fang--Zalcman \cite{Fang-Zalcman-2003} represents a milestone, revealing a striking phenomenon: the sharing of a carefully chosen finite set between an entire function and its derivative imposes a complete functional identity. Their result not only highlights the delicate interplay between value distribution and differentiation but also establishes the sharpness of the underlying conditions.

Before proceeding further, we recall some fundamental notions in several complex variables which will be used throughout the paper.

\medskip

Let $G\neq \varnothing$ be an open subset of $\mathbb{C}^n$. Let $f$ be a holomorphic function in $\mathbb{C}^n$.
Take $a\in G$. Let $G_a$ be the connectivity component of $G$ containing $a$. Assume $f\mid_{G_a}\not\equiv 0$. Then a series
$f(z)=\sum_{\lambda=p}^{\infty}P_{\lambda}(z-a)$ converges on some neighborhood of $a$ and represents $f$ on this neighborhood. Here 
$P_{\lambda}$ is a homogeneous polynomial of degree $\lambda$ and $P_{p}\not\equiv 0$. The polynomials $P_{\lambda}$ depend on $f$ and $a$ only. The number $\mu^0_f(a)=p$ is called the zero multiplicity of $f$ at $a$ (see \cite[pp. 12]{Stoll-1974}).

\medskip
Let $f$ be a meromorphic function on $G$.
Take $a\in G$ and $c\in\mathbb{C}\cup\{\infty\}$. Let $G_a$ be the component of $G$ containing $a$. If $0\equiv f\mid_{G_a}\not\equiv c$, define $\mu^c_f(a)=0$. Assume $0\not\equiv f\mid_{G_a}\not\equiv c$. Then an open connected neighborhood $U$ of $a$ in $G$ and holomorphic functions $g\not\equiv 0$ and $h\not\equiv 0$ exist on $U$ such that $h. f\mid_U=g$ and $\dim g^{-1}(0)\cap h^{-1}(0)\leq n-2$, where $n=\dim(\mathbb{C}^n)$. Therefore the $c$-multiplicity of $f$ is just $\mu^c_f=\mu^0_{g-ch}$ if $c\in\mathbb{C}$ and $\mu^c_f=\mu^0_h$ if $c=\infty$. The function $\mu^c_f:G\to \mathbb{Z}$ is nonnegative and is called the $c$-divisor of $f$ (see \cite[pp. 12]{Stoll-1974}).
If $f\not\equiv 0$ on each component of $G$, then $\nu=\mu_f=\mu^0_f-\mu^{\infty}_f$ is called the divisor of $f$. The function $f$ is holomorphic on $G$ if and only if $\mu_f\geq 0$. We define $\displaystyle \operatorname{supp}\; \nu=\ol{\{z\in G: \nu(z)\neq 0\}}$.

\medskip
Let $f$, $g$ and $a$ be meromorphic functions on $\mathbb{C}^n$. Then one can find three pairs
of entire functions $f_1$ and $f_2$, $g_1$ and $g_2$, and $a_1$ and $a_2$ in which each pair is coprime
at each point in $\mathbb{C}^n$ such that $f = f_2/f_1$, $g=g_2/g_1$ and $a = a_2/a_1$.
We say that $f$ and $g$ share $a$ CM if $\mu_{a_1f_2-a_2f_1}^0=\mu_{a_1g_2-a_2g_1}^0\;(a\not\equiv \infty)$ and $\mu_{f_1}^0=\mu_{g_1}^0\;\;(a=\infty)$. 

For $S\subset\mathbb{C}\cup\{\infty\}$, we write
\[E(S,f)=\bigcup\limits_{a\in S}\left\lbrace \mu^a_f(z): z\in\mathbb{C}^n\right\rbrace.\]
We say that $f$ and $g$ share the set $S$ CM if $E(S,f)=E(S,g)$.

\medskip

With these preliminaries in place, we now turn to the classical results that motivate the present investigation.

In $2003$, Fang--Zalcman \cite{Fang-Zalcman-2003} proved
\begin{theoA}\cite[Theorem 1]{Fang-Zalcman-2003} There exists a set $S$ containing three elements such that if $f$ is a non-constant entire function in $\mathbb{C}$ and $E(S,f)=E\left(S, f^{(1)}\right)$, then $f\equiv f^{(1)}$.
\end{theoA}

More specifically, Fang--Zalcman \cite{Fang-Zalcman-2003} showed that this holds if $S=\{0, a, b\}$, where $a$ and $b$ are distinct nonzero complex numbers such that $a^2 \neq b^2$, $a\neq 2b$, $2a\neq b$, and $a^2-ab+b^2\neq 0$. 

\smallskip
In the same paper, Fang--Zalcman \cite{Fang-Zalcman-2003} exhibited the following to show that the set contain three elements in Theorem A is best possible.
\begin{exmA} \cite{Fang-Zalcman-2003}
	Let $S=\{a,b\}$, where $a$ and $b$ are any two distinct complex numbers. Let $f(z)=e^{-z}+a+b$. Then $f^{(1)}(z)=-e^{-z}$. Obviously, $E(S, f)=E\left(S, f^{(1)}\right)$, but $f\not\equiv f^{(1)}$.
\end{exmA}

In $2007$, Chang et al. \cite{Chang-Fang-Zalcman-2007} extended Theorem A to an arbitrary set having three elements. We now recall their result.
\begin{theoB}\cite[Theorem 1]{Chang-Fang-Zalcman-2007} Let $f$ be a non-constant entire function in $\mathbb{C}$ and let $S=\{a, b, c\}$, where $a$, $b$ and $c$ are distinct complex numbers. If $E(S,f)=E\left(S, f^{(1)}\right)$, then either
	\begin{enumerate}
		\item[\emph{(i)}] $f(z)=Ce^z$ or
		\item[\emph{(ii)}] $f(z)=Ce^{-z}+\frac{2}{3}(a+b+c)$ and $(2a-b-c)(2b-c-a)(2c-a-b)=0$ or
		\item[\emph{(iii)}] $f(z)=Ce^{\frac{-1\pm i \sqrt{3}}{2}z}+\frac{3\pm i\sqrt{3}}{6}(a+b+c)$ and $a^2+b^2+c^2-ab-bc-ca=0$,
	\end{enumerate} 
	where $C$ is a non-zero constant in $\mathbb{C}$.
\end{theoB}

\medskip

\noindent
{\bf Transition to several variables.}
A natural and highly nontrivial question arises: to what extent do such rigidity phenomena persist in the setting of several complex variables? The transition from $\mathbb{C}$ to $\mathbb{C}^n$ is far from formal, as the geometry of zeros, multiplicities, and divisors becomes significantly more intricate. In particular, the interaction between partial derivatives and shared value sets introduces new structural constraints that have no analogue in the one-dimensional theory.

The main contribution of this paper is to demonstrate that the classical uniqueness paradigm extends, in a sharp and explicit form, to entire functions defined on $\mathbb{C}^n$. Remarkably, despite the increased complexity of the ambient space, the functional forms that arise remain rigid and exhibit a unified exponential structure.

\medskip

We define $\mathbb{Z}_+=\mathbb{Z}[0,+\infty)=\{m\in \mathbb{Z}: 0\leq m<+\infty\}$ and $\mathbb{Z}^+=\mathbb{Z}(0,+\infty)=\{m\in \mathbb{Z}: 0<m<+\infty\}$.
On $\mathbb{C}^n$, we define
\[\partial_{z_i}=\frac{\partial}{\partial z_i},\partial^2_{z_jz_i}=\frac{\partial^2}{\partial z_jz_i},\ldots, \partial_{z_i}^{l_i}=\frac{\partial^{l_i}}{\partial z_i^{l_i}}\;\;\text{and}\;\;\partial^{I}=\frac{\partial^{|I|}}{\partial z_1^{i_1}\cdots \partial z_n^{i_n}},\]
where $l_i\in \mathbb{Z}^+\;(i=1,2,\ldots,n)$ and $I=(i_1,\ldots,i_n)\in\mathbb{Z}^n_+$ is a multi-index such that $|I|=\sum_{j=1}^n i_j$.

\medskip

We are now in a position to state our main theorem.
\begin{theo}\label{t1.1}
	Let $f$ be a non-constant entire function in $\mathbb{C}^n$ and let $S=\{a,b,c\}$, where $a$, $b$,  $c$ $\in \mathbb{C}$. Suppose $E(S,f)=E\left(S,\partial_{z_i}(f)\right)$ for $i=1,2,\ldots,n$.  
\begin{enumerate}
	\item [\emph{(A)}] Suppose $a+b+c=0$. Then one of the following holds: 
	\begin{enumerate}
	\item[\emph{$(i)$}] If $abc\neq0$, then $f(z)=Ce^{a_{11}z_1+a_{12}z_2+\ldots+a_{1n}z_n}$ and $ab+bc+ca=0$, where $a_{11}, a_{12},\ldots,\\a_{1n}$ and $C$ are non-zero constants such that $a_{1i}^3=1$ for $i=1,2,\ldots,n$.
	\item[\emph{$(ii)$}] If $abc=0$, then $f(z)=Ce^{a_{11}z_1+a_{12}z_2+\ldots+a_{1n}z_n}$ and $ab+bc+ca\neq0$, where $a_{11}, a_{12},\ldots,\\a_{1n}$ and $C$ are non-zero constants such that $a_{1i}^2=1$ for $i=1,2,\ldots,n$;
	\end{enumerate}
	\item[\emph{(B)}] Suppose $a+b+c\neq 0$ and set $t=z_1+z_2+\cdots+z_n$. Then either
	\begin{enumerate}
		\item[\emph{$(i)$}] $f(z)=Ce^{t}$ or
		\item[\emph{$(ii)$}] $f(z)=Ce^{-t}+\frac{2}{3}(a+b+c)$ and $(2a-b-c)(2b-c-a)(2c-a-b)=0$ or
		\item[\emph{$(iii)$}] $f(z)=Ce^{\left(\frac{-1\pm i\sqrt{3}}{2}\right) t}+\frac{3\pm i\sqrt{3}}{6}(a+b+c)$
		and $a^2+b^2+c^2-ab-bc-ca=0$,
	\end{enumerate}
	where $C$ is a non-zero constant in $\mathbb{C}$.
\end{enumerate}
\end{theo}
\begin{rem}
	Considering the function $f(z)=c e^{\omega z_1+z_2^2+\cdots+z_n^2}$ and the set $S=\{1,\omega,\omega^{2}\}$, it is easy to verify that the condition $E(S,f)=E\left(S,\partial_{z_i}(f)\right)$ holds only for $i=1$, and not for any other index. Therefore, for the validity of Theorem \ref{t1.1}, it is necessary that the condition $E(S,f)=E\left(S,\partial_{z_i}(f)\right)$ holds for every $i=1,2,\ldots,n$.
\end{rem}
\begin{rem}
	Theorem \ref{t1.1} reveals an interesting feature in several complex variables. 
	Although an entire function $f$ on $\mathbb{C}^n$ can generally depend on all variables $z_1,\ldots,z_n$ in a complicated way, the condition $E(S,f)=E\left(S,\partial_{z_i}(f)\right), \quad i=1,2,\ldots,n,$
	forces a strong restriction: the function actually depends only on the single combination $z_1+z_2+\cdots+z_n$.
	
	As a result, the high-dimensional behaviour reduces to a one-dimensional one, indicating that the set-sharing condition enforces a dimensional reduction. 
\end{rem}
\subsection*{Geometric interpretation of Theorem \ref{t1.1}}
\par
The conclusions of Theorem \ref{t1.1} split into two qualitatively distinct regimes. 
Case (A) yields only exponential structures, whereas case (B) produces rich geometric dynamics in the complex plane.\\
\textbf{Analysis of Case (A):}\\
In case (A), the function $f(z)=C e^{a_{11}z_1+\cdots+a_{1n}z_n}, \quad a_{1i}^k=1, \quad k\in\{2,3\}$;
exhibits purely exponential growth along fixed directions. 
Thus, the geometry is essentially linear and rigid, and does not lead to additional structure in the $f$-plane.
\\
\textbf{Analysis of Case (B):}\\
\textbf{Case (i) (Pure exponential behaviour along a single direction):}\\
Here $f(z)=C e^{t}$, where $t=z_1+\cdots+z_n$. Since $f$ depends only on the single linear combination $t$, it is constant on the hyperplanes $z_1+\cdots+z_n = \text{constant}$.
Let $e=(1,\ldots,1)$ and write $v=\alpha e+w$, where $\alpha=\frac{1}{n}\sum_{i=1}^n v_i$ and $\sum_{i=1}^n w_i=0$, this implies $w\perp e$. Here, for a constant $h$ when $z=(z_1, z_2,\ldots, z_n)$ is replaced by $z+hv$, then $t$ is replaced by $t+h\sum v_i=t+hn\alpha$, so only the component along $e$ changes $t$ and hence $f=Ce^t$ varies only in the direction $e$ in $\mathbb{C}^n$.\\
\textbf{Case (ii) (Exponential decay with shift):}\\ Here $f(z)=Ce^{-t}+\frac{2}{3}(a+b+c)$, where $t=z_1+\cdots+z_n$. Consequently, $f$ again depends only on $t$, showing exponential decay along $(1,1,\ldots,1)$ together with a constant shift. Here, the condition over $a$, $b$, and $c$ implies collinearity in $\mathbb{C}$.\\ %The condition $(2a-b-c)(2b-c-a)(2c-a-b)=0$ implies that one of
%$2a=b+c,\quad 2b=c+a,\quad 2c=a+b$
%holds, i.e., one of $a,b,c$ is the average of the other two; hence $a,b,c$ are collinear in $\mathbb{C}$. \\
\textbf{Case (iii) (Oscillatory exponential behaviour):}\\
Here $f(z)=C e^{\lambda t}+\text{constant}$, $t=z_1+\cdots+z_n,\quad 
\lambda=\frac{-1\pm i\sqrt{3}}{2}$. We can write
$e^{\lambda t}= e^{\Re(\lambda t)}\, e^{i\,\Im(\lambda t)}.
$ Let $t=x+iy$. Considering $\lambda=\frac{-1+ i\sqrt{3}}{2}$, we can have $\Re(\lambda t)
= -\tfrac{1}{2}x - \tfrac{\sqrt{3}}{2}y, 
\quad
\Im(\lambda t)
= \tfrac{\sqrt{3}}{2}x - \tfrac{1}{2}y.
$
Hence,
$
|e^{\lambda t}| = e^{\Re(\lambda t)}
= e^{-\frac{1}{2}x - \frac{\sqrt{3}}{2}y},
$
which shows that the magnitude exhibits direction-dependent exponential growth or decay, depending on the sign of $\Re(\lambda t)$, while
$e^{i\,\Im(\lambda t)}$
produces oscillatory (rotational) behaviour. The case $\lambda=\frac{-1- i\sqrt{3}}{2}$ can similarly be treated analogously. 

 Along the direction, $(1,1,\ldots,1)$ in $\mathbb{C}^n$; $f$ exhibits oscillatory exponential behaviour, with growth or decay determined by $\Re(\lambda t)$ and rotation in the complex plane governed by $\Im(\lambda t)$.

Further, the condition imposed upon $a,b,c$ are equilateral triangle in the complex plane.
 %$a^2+b^2+c^2-ab-bc-ca=0 
%\;\Rightarrow\; 
%(a-b)^2+(b-c)^2+(c-a)^2=0,$
%

\noindent

Consequently, the $n$-dimensional behaviour reduces to a geometric motion in the complex $f$-plane,  parametrized by the real and imaginary parts of $\lambda t$ summarized as follows: 

\begin{table}[h]
	\centering
	\renewcommand{\arraystretch}{1.4} 
	\setlength{\tabcolsep}{8pt}       
	\begin{tabularx}{\textwidth}{
		!{\vrule width 1.5pt}
		>{\centering\arraybackslash}p{2cm}
		!{\vrule width 1.5pt}
	>{\centering\arraybackslash}p{2cm}
		!{\vrule width 1.5pt}
		X
		!{\vrule width 1.5pt}}
		\noalign{\hrule height 1.2pt} 
		\rowcolor{gray!20}
		\textbf{Regime} & \textbf{Analytic Condition} & \textbf{\hspace{5cc}Geometric Dynamics} \\
		\noalign{\hrule height 1.2pt} 
		
		Pure oscillation 
		& $\Re(\lambda t)=0$ 
		& $|f-d|=|C|$ is constant; trajectory is confined to a circle centered at $d$, governed by rotation in $\Im(\lambda t)$ \\
		\noalign{\hrule height 1.2pt} 
		
		Decay 
		& $\Re(\lambda t)<0$ 
		& Radius $|f-d|=|C| e^{\Re(\lambda t)}$ decreases; trajectories form inward spirals approaching $d$ \\
		\noalign{\hrule height 1.2pt}
		
		Growth 
		& $\Re(\lambda t)>0$ 
		& Radius $|f-d|=|C| e^{\Re(\lambda t)}$ increases; trajectories form outward spirals moving away from $d$ \\
		\noalign{\hrule height 1.2pt} 
	\end{tabularx}
	\vspace{1cc}
	\caption{Geometric behaviour of $f=C e^{\lambda t}+d$ in the $f$-plane. 
		For fixed $\Re(\lambda t)$, trajectories are circles, while variation in $\Re(\lambda t)$ produces spiral motion.}
\end{table}

\begin{figure}[h]
	\centering
	
	\begin{center}
		
		\tikzset{
			axis/.style={line width=1.8pt, black, ->, >=stealth}, % thicker axes
			maincurve/.style={line width=1.2pt},
			auxcurve/.style={line width=0.9pt, dashed},
			innercurve/.style={line width=0.8pt, dotted}
		}
		
		\begin{minipage}[t]{0.43\textwidth}
			\centering
			\begin{tikzpicture}[scale=0.66, baseline=(current bounding box.south)]
				
				\path (-4.8,-4.8) rectangle (4.8,4.8);
								
				\draw[axis] (-4.5,0) -- (4.5,0) node[right] {$x$};
				\draw[axis] (0,-4.5) -- (0,4.5) node[above] {$y$};
								
				\foreach \c in {-3,-2,-1,0,1,2,3} {
					\draw[maincurve, blue!70] 
					(-4,{(-2*\c-1*(-4))/sqrt(3)}) -- 
					(4,{(-2*\c-1*(4))/sqrt(3)});
				}
				
			\end{tikzpicture}
		\end{minipage}
		\hspace{0.03\textwidth}
		\begin{minipage}[t]{0.45\textwidth}
			\centering
			\begin{tikzpicture}[scale=0.78, baseline=(current bounding box.south)]
				
			\path (-4.8,-4.8) rectangle (4.8,4.8);

				\draw[axis] (-4.5,0) -- (4.5,0) node[right] {$\Re f$};
				\draw[axis] (0,-4.5) -- (0,4.5) node[above] {$\Im f$};
				
				\coordinate (d) at (1,0.8);

				\fill[red!70, opacity=0.25] (d) circle (5pt);
				\fill[red!80, opacity=0.4] (d) circle (3pt);
				\fill[red!100] (d) circle (1.6pt);

				\node[
				font=\fontsize{13}{13}\selectfont\bfseries,
				text=black
				] at ($(d)+(0.75,0.07)$) {$d$};

				\foreach \r in {0.15,0.3,0.4} {
					\draw[innercurve, red!30] (d) circle (\r);
				}

				\foreach \r in {0.5,1,1.5,2} {
					\draw[maincurve, red!70] (d) circle (\r);
				}

				\foreach \r in {2.5,3,3.5} {
					\draw[auxcurve, red!70] (d) circle (\r);
				}
				
			\end{tikzpicture}
		\end{minipage}
		
		\vspace{0.25cm}
		
		{{\footnotesize Left: Level lines of $\Re(\lambda t)$ in the $t$-plane 
				($t=z_1+\cdots+z_n$), shown as parallel straight lines. 
				Right: Their images under $f=C e^{\lambda t}+d$ form circles centered at 
				$d=\frac{3\pm i\sqrt{3}}{6}(a+b+c)$ in the $f$-plane. 
				Solid curves represent finite levels, while dashed curves indicate limiting behaviour.}}
		
	\end{center}
	
	\caption{Geometric correspondence between level lines in the $t$-plane and their images in the $f$-plane.}
	\label{fig:level-lines}
	
\end{figure}

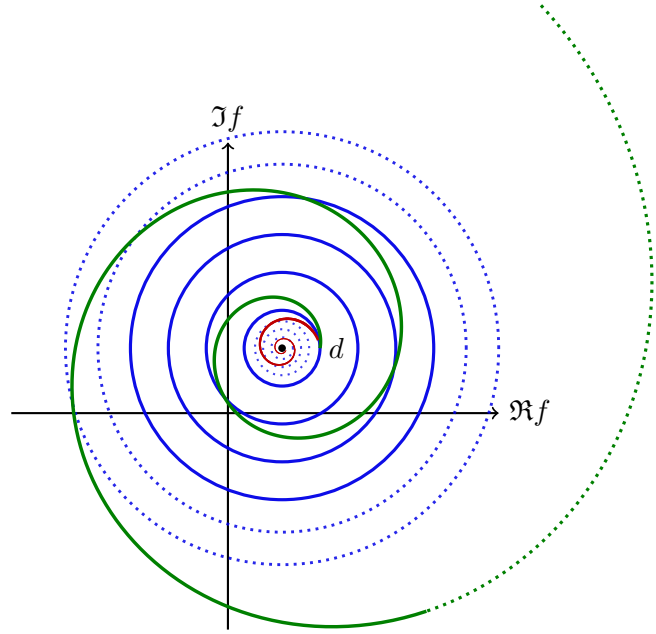
\begin{figure}[h]
	\centering
	\begin{tikzpicture}[scale=0.72, baseline=(current bounding box.center)]

		\colorlet{myblue}{blue!90!black}
		\colorlet{myred}{red!75!black}
		\colorlet{mygreen}{green!50!black}

		\draw[->, thick] (-4,0) -- (5,0) node[right] {$\Re f$};
		\draw[->, thick] (0,-4) -- (0,5) node[above] {$\Im f$};

		\coordinate (d) at (1,1.2);
		\fill[black] (d) circle (2pt);
		\node[font=\bfseries] at ($(d)+(1.0,0.01)$) {$d$};

		\foreach \r in {0.7,1.4,2.1,2.8} {
			\draw[myblue, line width=1.2pt, opacity=0.95] (d) circle (\r);
		}

		\foreach \r in {0.2,0.35,0.5} {
			\draw[myblue, dotted, line width=0.8pt, opacity=0.75] (d) circle (\r);
		}

		\foreach \r in {3.4,4.0} {
			\draw[myblue, dotted, line width=1.1pt, opacity=0.85] (d) circle (\r);
		}

	\foreach \k in {0,...,40} {
		\pgfmathsetmacro{\tstart}{6.5 + 0.28*\k}
		\pgfmathsetmacro{\tend}{6.9 + 0.28*(\k+1)}
		\pgfmathsetmacro{\lw}{1.6*(1 - \tstart/18)}
		
		\draw[myred, line width=\lw pt]
		plot[variable=\t, domain=\tstart:\tend, samples=20]
		({1 + 2.2*exp(-0.18*\t)*cos(\t r)},
		{1.2 + 2.2*exp(-0.18*\t)*sin(\t r)});
	}

		\def\rzero{0.7}

		\draw[mygreen, line width=1.3pt, samples=300, variable=\t, domain=0:11.5]
		plot ({1 + \rzero*exp(0.18*\t)*cos(\t r)}, {1.2 + \rzero*exp(0.18*\t)*sin(\t r)});

		\draw[mygreen, dotted, line width=1.2pt, samples=200, variable=\t, domain=11.5:13.5]
		plot ({1 + \rzero*exp(0.18*\t)*cos(\t r)}, {1.2 + \rzero*exp(0.18*\t)*sin(\t r)});
		
	\end{tikzpicture}
	
	\caption{Blue circles indicate trajectories for fixed $\Re(\lambda t)$ and the first circle corresponding to $\Re(\lambda t)=0$. 
		The red inward spiral shows decay, while the green outward spiral shows growth. 
		Rotation around $d$ is governed by $\Im(\lambda t)$, producing the spiraling motion.}
\end{figure}
\section{\bf{Basic Notations in Several Complex Variables}}

Nevanlinna theory, originally developed for meromorphic functions of a single complex variable, has undergone profound generalization to the setting of several complex variables. In this broader framework, it investigates the value distribution of meromorphic mappings from $\mathbb{C}^n$ into complex projective spaces or, more generally, complex manifolds. This theory serves as a powerful analytical tool, yielding deep insights into:

\begin{itemize}
	\item[(i)] the growth and value distribution of meromorphic mappings,
	\item[(ii)] uniqueness and rigidity phenomena for solutions of linear and nonlinear partial differential equations,
	\item[(iii)] compactness criteria and convergence properties of families of meromorphic maps.
\end{itemize}

Its scope of application extends across complex analysis, partial differential equations, complex geometry, and mathematical physics, providing both a unifying perspective and a rich source of techniques. The works \cite{BM}, \cite{Cao-Korhonen-2016}, \cite{CL}, \cite{PVD}-\cite{PVD3}, \cite{BQL1}-\cite{FL1}, \cite{MDP}, \cite{Majumder-2026}, \cite{Majumder-Sarkar-2027}, \cite{GS2} collectively offer a comprehensive account of the current developments in Nevanlinna value distribution theory in several complex variables.

\medskip

Let $A\subseteq \mathbb{C}^n$ and $r\geq 0$. Following \cite[pp.~6]{Stoll-1974}, we define
\[
A[r]=\{z\in A: \|z\|\leq r\}, \quad 
A(r)=\{z\in A: \|z\|<r\}, \quad 
A\langle r\rangle=\{z\in A: \|z\|=r\},
\]
and the exhaustion function $\tau: \mathbb{C}^n \to \mathbb{R}_+$ by $\tau(z)=\|z\|^2$.

On $\mathbb{C}^n$, the exterior derivative admits the decomposition
\[
d=\partial+\bar{\partial}, \qquad 
d^c=\frac{i}{4\pi}(\bar{\partial}-\partial),
\]
so that
\[
dd^c=\frac{i}{2\pi}\partial\bar{\partial}.
\]
The standard Kähler form on $\mathbb{C}^n$ is given by
\[
\upsilon = dd^c \tau > 0.
\]
On $\mathbb{C}^n\backslash\{0\}$, we further define
\[
\omega = dd^c \log \tau \geq 0, 
\qquad 
\sigma = d^c \log \tau \wedge \omega^{n-1},
\]
where $n=\dim(\mathbb{C}^n)$.

\medskip

For $t>0$, the counting function associated with a divisor $\nu$ is defined by
\begin{align*}
	n_{\nu}(t)=t^{-2(n-1)}\int_{A[t]} \nu\, \upsilon^{n-1},
\end{align*}
where $A=\operatorname{supp}\nu$. The corresponding integrated counting (or valence) function is given by
\[
N_{\nu}(r)=\int_{r_0}^r n_{\nu}(t)\frac{dt}{t}, \qquad (r\geq r_0).
\]

\medskip

We adopt the notation
\[
N_{\mu_f^a}(r)=N(r,a;f)\quad \text{for } a\in\mathbb{C}, 
\qquad 
N_{\mu_f^\infty}(r)=N(r,f).
\]

\smallskip

Using the positive logarithm $\log^+ x=\max\{\log x,0\}$, the proximity function of $f$ is defined as
\begin{align*}
	m(r,f)=\int_{\mathbb{C}^n\langle r\rangle} \log^+|f| \;\sigma.
\end{align*}

The characteristic function of $f$ is then given by
\[
T(r,f)=m(r,f)+N(r,f).
\]
For $a\in\mathbb{C}$, we define
\[
m(r,a;f)=
\begin{cases}
	m(r,f), & a=\infty,\\
	m\big(r, \frac{1}{f-a}\big), & a\in\mathbb{C}.
\end{cases}
\]
The First Main Theorem asserts that
\[
m(r,a;f)+N(r,a;f)=T(r,f)+O(1),
\]
where $O(1)$ denotes a bounded quantity for sufficiently large $r$.

Finally, the order of growth of $f$ is defined by
\begin{align*}
	\rho(f):=\limsup_{r\to\infty}\frac{\log T(r,f)}{\log r}.
\end{align*}

\section{\bf{Auxiliary Lemmas}}
For every function $\varphi$ of class $C^2(\Omega)$, we define at each point $z\in\Omega$ an hermitian form (see \cite{PVD1,PVD2})
\begin{align*}
L_z(\varphi,\nu)=\sideset{}{_{l,k=1}^n}{\sum} \frac{\partial^2 \varphi(z)}{\partial z_k \partial \ol{z}_l}\nu_k \ol{\nu}_l,
\end{align*}
which is called the Levi form of the function $\varphi$. For a holomorphic function $f$ in $\Omega$, we define
\begin{align*}
f^{\#}(z)=\max_{||\nu||=1}\sqrt{L_z(\log (1+|f(z)|^2),\nu)},
\end{align*}
where $\nu=(\nu_1,\ldots,\nu_n)$. This quantity is well defined, since the Levi form $L_z(\log (1+|f(z)|^2),\nu)$ is non-negative for all $z\in\Omega$. Let ${\bf \nabla} f=(f_{z_1},\ldots,f_{z_n})$.

Applying Cauchy-Schwarz inequality, it is easy to prove that (see \cite[Remark 1]{ZZY1})
\begin{align*}
f^{\#}(z)=\sup_{||\nu||=1}\frac{|\langle {\bf \nabla} f(z),\nu\rangle|}{1+|f(z)|^2}=\frac{||{\bf \nabla} f(z)||}{1+|f(z)|^2}=\frac{\sqrt{\sum_{j=1}^n|f_{z_j}(z)|^2}}{1+|f(z)|^2},\;\;z\in\Omega,
\end{align*}
where $\langle z, w\rangle=\sum_{j=1}^n z_j\ol{w}_j$ is the Hermitian scalar product for $z=(z_1,\ldots,z_n)\in\mathbb{C}^n$ and $w=(w_1,\ldots,w_n)\in\mathbb{C}^n$.
\medskip

A family $\mathcal{F}$ of holomorphic functions on a domain $\Omega\subset \mathbb{C}^n$ is normal in $\Omega$ if every sequence of functions $\{f_j\}\subseteq \mathcal{F}$ contains either a subsequence which converges to a limit function $f\neq \infty$ uniformly on each compact subset of $\Omega$, or a subsequence which converges uniformly to $\infty$ on each compact subset.

A family $\mathcal{F}$ is said to be normal at a point $z_0\in \Omega$ if it is normal in some neighbourhood of $z_0$. A family of analytic functions $\mathcal{F}$ is normal in a domain $\Omega$ if and only if $\mathcal{F}$ is normal at each point of $\Omega$.

\smallskip
We recall the Marty's \cite{PVD1} characterization of normal families in terms of the spherical metric.

\begin{lem}\label{ln2.1}\cite[Theorem 2.1]{PVD1} A family $\mathcal{F}$ of functions holomorphic on $\Omega$ is normal on $\Omega\subset \mathbb{C}^n$ if and only if for each compact subset $K\subset \Omega$, there exists $M(K)$ such that at each point $z\in K$,
$f^{\#}(z)\leq M(K)$ $\forall$ $f\in\mathcal{F}$.
\end{lem} 

In 2021, Dovbush \cite{PVD2} obtained an appropriate generalization of Zalcman's Lemma \cite{LZ1} to several complex variables and the result is as follows.
\begin{lem}\label{ln2.2} \cite[Theorem 1.1]{PVD2} Suppose that a family $\mathcal{F}$ of functions holomorphic on $\Omega\subset \mathbb{C}^n$ is not normal at some point $z_0\in\Omega$. Then there exist sequences $f_j\in\mathcal{F}$, $z_j\to z_0$, $\rho_j\to 0$ such that the
sequence
\begin{align*}
g_j(z):=\rho_j^{-\alpha}f_j(z_j+\rho_jz)\;\;(0\leq \alpha<1\;\;\text{arbitrary})
\end{align*}
converges locally uniformly in $\mathbb{C}^n$ to a non-constant entire function $g$ satisfying $g^{\#}(z)\leq g^{\#}(0)=1$.
\end{lem}

In \cite{Majumder-2026}, Majumder studied the relation between the maximum modulus and the spherical metric of a holomorphic function in $\mathbb{C}^n$ and obtained the following result.
\begin{lem}\cite{Majumder-2026}\label{ln2.3} Let $f$ be a holomorphic function in $\mathbb{C}^n$. If $f^{\#}$ is bounded on $\mathbb{C}^n$, then $\rho(f)\leq 1$.
\end{lem}

\begin{lem}\label{ln2.4}\cite[Lemma 1.37]{Hu-Li-Yang-2003} Let $f$ be a non-constant meromorphic function in $\mathbb{C}^n$ and let $I=(i_1,\ldots,i_n)\in \mathbb{Z}^n_+$ be a multi-index. Then for any $\varepsilon>0$, we have
\begin{align*}
m\left(r,\frac{\partial^I(f)}{f}\right)\leq |I|\log^+T(r,f)+|I|(1+\varepsilon)\log^+\log T(r,f)+O(1),
\end{align*}
for all large $r$ outside a set $E$ with $\text{log mes}\;E=\int_E d\log r<\infty$.
\end{lem}

\begin{lem}\label{ln2.5}\cite[Lemma 1.2]{Hu-Yang-1996} Let $f$ be a non-constant meromorphic function in $\mathbb{C}^n$ and let $a_1,a_2,\ldots,a_q$ be different points in $\mathbb{C}\cup\{\infty\}$. Then
\begin{align*}
(q-2)T(r,f)\leq \sideset{}{_{j=1}^{q}}{\sum} \ol N(r,a_j;f)+O(\log (rT(r,f)))
\end{align*}
holds only outside a set of finite measure on $\mathbb{R}^+$.
\end{lem}

\begin{lem}\label{ln2.6}\cite[Lemma 3.59]{Hu-Li-Yang-2003} Let $P$ be a non-constant entire function in $\mathbb{C}^n$. Then 
\begin{align*}
\rho(e^P)=
\begin{cases}
\deg(P), & \text{if $P$ is a polynomial,}\\
+\infty, & \text{otherwise.}
\end{cases}
\end{align*}
\end{lem}

Let $f$ and $g$ be two non-constant holomorphic functions in a domain $D\subset \mathbb{C}^n$ and let $a\in\mathbb{C}$. If $g(z)=a$ whenever $f(z)=a$, we write $f= a\Rightarrow g=a$. If $f=a\Rightarrow g=a$ and $g=a\Rightarrow f=a$, we write $f=a \Leftrightarrow g=a$.

\begin{lem}\label{ln2.7} Let $\mathcal{F}$ be a family of holomorphic functions on a domain $\Omega\subset \mathbb{C}^n$ and let $a$, $b$ and $c$ be three distinct finite complex numbers. For each $f\in \mathcal{F}$, if $E(S,f)=E(S,\partial_{z_i}(f))$ for $i=1,2,\ldots,n$, then $\mathcal{F}$ is normal in $\Omega$.
\end{lem}
\begin{proof}
Since normality is a local property, it is enough to show that $\mathcal{F}$ is normal at each point $z_0\in \Omega$.
Suppose on the contrary that $\mathcal{F}$ is not normal at $z_0\in\Omega$. 
Then by Lemma \ref{ln2.2}, there exist a sequence of functions $f_j\in\mathcal{F}$, a sequence $\{z_j\}\subset\Omega$ with $z_j\rightarrow z_0$ and a sequence of positive numbers $\{\rho_j\}$ with $\rho_j\rightarrow 0$ such that
\bea\label{rj1.1} F_j(\zeta)=f_j(z_j+\rho_j \zeta)\rightarrow F(\zeta),\eea
locally uniformly in $\mathbb{C}^n$ to a non-constant holomorphic function $F$ such that $F^{\#}(\zeta)\leq F^{\#}(0)=1$, $\forall$ $\zeta\in\mathbb{C}^n$. Since $F^{\#}(\zeta)\leq 1$ for all $\zeta\in\mathbb{C}^n$, by Lemma \ref{ln2.3}, we get $\rho(F)\leq 1$. Also (\ref{rj1.1}) gives
\bea\label{rj1.2} F(\zeta)-w=\lim\limits_{j\rightarrow \infty}\left(F_j(\zeta)-a\right)=\lim\limits_{j\rightarrow \infty} \left(f_j(z_j+\rho_j \zeta)-w\right),\eea
locally uniformly in $\mathbb{C}^n$, where $w\in S$ and
\bea\label{rj1.3} \partial_{\zeta_i}(F_j(\zeta))=\rho_j\partial_{\zeta_i}(f_j(z_j+\rho_j \zeta))\rightarrow \partial_{\zeta_i}(F(\zeta)),\eea
locally uniformly in $\mathbb{C}^n$, where $i\in\{1,2,\ldots,n\}$. Note that
\bea\label{xx} F^{\#}(\zeta)=\frac{\sqrt{\sum\limits_{j=1}^n|\partial_{\zeta_j}(F(\zeta))|^2}}{1+|F(\zeta)|^2},\;\;\zeta\in\mathbb{C}^n\eea
 and $F^{\#}(\zeta)\leq F^{\#}(0)=1$ for all $\zeta\in\mathbb{C}^n$. Consequently from (\ref{xx}), it is easy to conclude that $\partial_{\zeta_k}(F(\zeta))\not\equiv 0$ for atleast one $k\in\{1,2,\ldots,n\}$.

\smallskip
If $F\neq a$, then obviously $F=a\Rightarrow \partial_{\zeta_k}(F)=0$. Next let $F(\zeta_0)=a$. Then applying Hurwitz's theorem (see \cite{PVD3}) to (\ref{rj1.2}), we get a sequence $\{\zeta_j\}$ in $\mathbb{C}^n$ such that $\zeta_j\rightarrow \zeta_0$ and $F_j(\zeta_j)=f_j(z_j+\rho_j \zeta_j)=a$. Since $E(S,f_j(z))=E(S,\partial_{\zeta_k}(f_j(z)))$ for all $j$, we have 
\[|\partial_{\zeta_k}(f_j(z))|\leq \max\{|a|,|b|,|c|\}=M_1\]
for all $j$. Then from (\ref{rj1.3}), we have
\[|\partial_{\zeta_k}(F(\zeta_0))|=\lim\limits_{j\rightarrow \infty}|\partial_{\zeta_k}(F_j(\zeta_j))|\leq \lim\limits_{j\rightarrow \infty} \rho_j M_1=0,\]
which shows that $\partial_{\zeta_k}(F(\zeta_0))=0$.
Hence $F=a\Rightarrow \partial_{\zeta_k}(F)=0$. Similarly, $F=b\Rightarrow \partial_{\zeta_k}(F)=0$ and $F=c\Rightarrow \partial_{\zeta_k}(F)=0$. 

\smallskip
Now in view of the first main theorem and using Lemmas \ref{ln2.4} and \ref{ln2.5}, we get
\begin{align*}2 T(r,F)&\leq \ol N(r,a;F)+\ol N(r,b;F)+\ol N(r,c;F)+o(T(r,F))\\
&\leq N(r,0;\partial_{\zeta_k}(F))+o(T(r,F))\\
&\leq T(r,\partial_{\zeta_k}(F))+o(T(r,F))\\
&\leq m(r,\partial_{\zeta_k}(F))+o(T(r,F))\\
&\leq m(r,F)+o(T(r,F))\\
&\leq T(r,F)+o(T(r,F)),
\end{align*}
which is impossible.
Hence $\mathcal{F}$ is normal at $z_0$. Consequently $\mathcal{F}$ is normal on $\Omega$. 
This completes the proof.
\end{proof}

Here, conclusion of this lemma does not hold when the set contains only two elements. The following example illustrate this fact:
\begin{exm}
Let us choose the set $S=\{1, -1\}$ and family of holomorphic functions
\begin{align*}
\mathcal{F}=\left\lbrace f_j(z): \ f_j(z)=\frac{j+1}{2j}e^{j(z_1+z_2+\cdots+z_n)}+\frac{j-1}{2j}e^{-j\left(z_1+z_2+\cdots+z_n\right)}, \ j=2, 3, \cdots\right\rbrace
\end{align*}
on $\Omega=\{z\in\mathbb{C}^n: \|z\|<1\}$. Now, for any $f_j\in \mathcal{F}$, we have 
\begin{align*}
j^2\left(f_j^2(z)-1\right) =\left(\frac{\partial f_j(z)}{\partial z_i}\right)^2-1.
\end{align*}

This shows that $E(S,f_j)=E\left(S,\partial_{z_i}(f_j)\right)$ for $i=1,2,\ldots,n$ and $j=2,3,\ldots$, but the family $\mathcal{F}$ is not a normal family.
\end{exm}

For the concavity of logarithmic function, we have the following lemma.
\begin{lem}\label{ln2.7a}\cite[Lemma 1.32]{Hu-Li-Yang-2003} Take $r>0$. Let $h$ be a non-negative function on $\mathbb{C}^n(r)$ such that $\log^+ h$ is integrable over $\mathbb{C}^n(r)$. Then 
\[\mathbb{C}^n[r;\log^+h]\leq \log^+(\mathbb{C}^n[r;h])+\log 2.\]
\end{lem}

In 1995, Ye \cite{Ye-1995} obtained the following result.
\begin{lem}\cite[Lemma 4]{Ye-1995} \label{ln2.7aa} Let $f$ be a non-constant meromorphic function in $\mathbb{C}^n$. Then for any $0<\alpha<\frac{1}{2}$, there is a constant $C>1$ such that for any $r_0<r<R$ and any $j\in\{1,2,\ldots,n\}$ we have 
\[\mathbb{C}^n\bigg\langle r;\bigg|\frac{\partial_{z_j}(f)}{f}\bigg|^{\alpha}\bigg\rangle\leq C\left\{\left(\frac{R}{r}\right)^{2n-1}\frac{T(R,f)}{R-r}\right\}^{\alpha}.\]
\end{lem}

\begin{lem}\label{ln2.8} Let $f$ be a non-constant meromorphic function on $\mathbb{C}^n$ such that $\rho(f)\leq 1$. Then any $j\in\{1,2,\ldots,n\}$, we have
\begin{align*}
m\left(r,\frac{\partial_{z_j}(f)}{f}\right)=o(\log r)\;\;\text{as}\;\;r\to \infty.
\end{align*}
\end{lem}
\begin{proof} By the definition of the proximity function, we have
\begin{align}\label{aa.1}
 m\left(r,\frac{\partial_{z_j}(f)}{f}\right)=\mathbb{C}^n\bigg\langle r;\log^+\bigg|\frac{\partial_{z_j}(f)}{f}\bigg|\bigg\rangle,
 \end{align}
where $j\in\{1,2,\ldots,n\}$. Now using Lemmas \ref{ln2.7a} and \ref{ln2.7aa} to (\ref{aa.1}), we obtain 
\begin{align}\label{aa.2} 
m\left(r,\frac{\partial_{z_j}(f)}{f}\right)&\leq 
\frac{1}{\alpha}\log^+ \mathbb{C}^n\bigg\langle r;\log^+\bigg|\frac{\partial_{z_j}(f)}{f}\bigg|^{\alpha}\bigg\rangle+O(1)\nonumber\\
&\leq
 \log^+\left(\left(\frac{R}{r}\right)^{2n-1}\frac{T(R,f)}{R-r}\right)+O(1),
 \end{align}
where $R>r>r_0$ and $j\in\{1,2,\ldots,n\}$. Again by the given condition for every $\varepsilon>0$, there exists $r_1>0$ such that 
\begin{align}\label{aa.3} 
T(r,f)<r^{\rho+\varepsilon},\;\forall \;\; r>r_1.
\end{align}

Let us choose $R=2r$, where $r>\max\{r_0,r_1\}$. Then from (\ref{aa.2}) and (\ref{aa.3}), we get 
\begin{align}\label{aa.4} 
m\left(r,\frac{\partial_{z_j}(f)}{f}\right)&\leq
\log^+(2^{2n+\rho+\varepsilon-1}r^{\rho+\varepsilon-1})+O(1)\nonumber\\&\leq
 \log^+(r^{\rho+\varepsilon-1})+O(1),
 \end{align}
where $j\in\{1,2,\ldots,n\}$. 

If $\rho<1$, then choosing $\varepsilon>0$ in such a way that $\rho+\varepsilon-1<0$, we get from (\ref{aa.4}) that $m\left(r,\frac{\partial_{z_j}(f)}{f}\right)=O(1)$ and so
\begin{align*}
m\left(r,\frac{\partial_{z_j}(f)}{f}\right)=o(\log r)\;\;\text{as}\;\;r\to \infty,
\end{align*}
where $j\in\{1,2,\ldots,n\}$.
Next suppose that $\rho=1$. Then from (\ref{aa.4}), we get
\bea\label{aa.5} m\left(r,\frac{\partial_{z_j}f}{f}\right)&\leq & \log^+(2^{2n+\varepsilon}r^{\varepsilon})+O(1)
 \leq \log^+(r^{\varepsilon})+O(1)=\varepsilon \log r+O(1)\eea
for sufficiently large $r$. Since $\varepsilon>0$ is arbitrary, we get from (\ref{aa.5}) that
\begin{align*}
 m\left(r,\frac{\partial_{z_j}(f)}{f}\right)=o(\log r)\;\;\text{as}\;\;r\to\infty,
 \end{align*}
 for $j\in\{1,2,\ldots,n\}$.
\end{proof}

\begin{lem}\label{ln2.9} Let $a$, $b$ and $c$ be three distinct finite complex numbers and $\alpha_i$ be an entire function in $\mathbb{C}^n$ for $i=1,2,\ldots,n$. Suppose an entire function $f$ in $\mathbb{C}^n$ satisfies
the partial differential equation
\begin{align}\label{MBP1}
\frac{(\partial_{z_i}(f)-a)(\partial_{z_i}(f)-b)(\partial_{z_i}(f)-c)}{(f-a)(f-b)(f-c)}=e^{\alpha_i}
\end{align}
for $i=1,2,\ldots,n$. If atlaest one of $\alpha_1$, $\alpha_2,\ldots$, $\alpha_n$ is non-constant, say $\alpha_k$, then only one of the following cases holds
\begin{enumerate}
\item[\emph{(i)}] $f(\tilde z)=c_ke^z$, where $\tilde z=(z_1,\ldots,z_{k-1},z,z_{k+1},\ldots,z_n)$ such that $z_j=0$ for $j\neq k$ and $c_k$ is a non-zero constant in $\mathbb{C}$;
\item[\emph{(ii)}] $f(\tilde z)=c_ke^{-z}+\frac{2}{3}(a+b+c)$ and $(2a-b-c)(2b-c-a)(2c-a-b)=0$, where $\tilde z=(z_1,\ldots,z_{k-1},z, z_{k+1},\ldots,z_n)$ such that $z_j=0$ for $j\neq k$ and $c_k$ is a non-zero constant in $\mathbb{C}$;
\item[\emph{(iii)}] $f(\tilde z)=c_ke^{\frac{-1\pm i\sqrt{3}}{2}z}+\frac{3\pm i\sqrt{3}}{6}(a+b+c)$ and $a^2+b^2+c^2-ab-bc-ca=0$, where $\tilde z=(z_1,\ldots,z_{k-1},z,z_{k+1},\ldots,z_n)$ such that $z_j=0$ for $j\neq k$ and $c_k$ is a non-zero constant in $\mathbb{C}$.
\end{enumerate} 
\end{lem}
\begin{proof} Suppose that there exists an entire function $f$ in $\mathbb{C}^n$ satisfying the partial differential equation (\ref{MBP1}). Also from (\ref{MBP1}), it is easy to deduce that $E(S,f)=E\left(S,\partial_{z_i}(f)\right)$ for $i=1,2,\ldots,n$ and so $|\partial_{z_i}(f(z))|\leq \max\{|a|, |b|, |c|\}$ whenever $f(z)\in\{a, b, c\}$ for $i=1,2,\ldots,n$. Thus, by Lemma \ref{ln2.7}, the family $\{f_w(z)=f(z + w): w\in\mathbb{C}^n\}$ is normal in $\Omega\subset \mathbb{C}^n$ and so by Lemma \ref{ln2.1}, $f^{\#}(w)=\left(f_w\right)^{\#}(0)$ is uniformly bounded for all $w\in \mathbb{C}^n$. Hence, by Lemma \ref{ln2.3}, $f$ has order at most $1$. Therefore, using Lemma \ref{ln2.6} to (\ref{MBP1}), we conclude that $\alpha_i(z) =A_{i1}z_1+A_{i2}z_2+\ldots+A_{in}z_n+B_i$, where $A_{i1},A_{i2},\ldots,A_{in},B_{i}$ are constants such that $(A_{i1},A_{i2},\ldots,A_{in})\neq (0,0,\ldots,0)$.
Suppose $\alpha_k(z)=A_{k1}z_1+A_{k2}z_2+\ldots+A_{kn}z_n+B_k$ is non-constant. Then clearly $A_{k1},A_{k2},\ldots,A_{kn},B_{k}$ are constants such that $(A_{k1},A_{k2},\ldots,A_{kn})\neq (0,0,\ldots,0)$. It is easy to conclude from (\ref{MBP1}) that $f$ is a transcendental entire function.

We now consider the following two cases.

\smallskip
{\bf Case 1.} Let $abc=0$. Since $a$, $b$ and $c$ are distinct, for the sake of simplicity, we may assume that $c=0$.  It follows from (\ref{MBP1}) that
\begin{align}\label{M1}
\frac{\partial_{z_k}(f)\left(\partial_{z_k}(f)-a\right)\left(\partial_{z_k}(f)-b\right)}{f\left(f-a\right)\left(f-b\right)}=e^{A_{k1}z_1+A_{k2}z_2+\ldots+A_{kn}z_n+B_k}.
\end{align}

Again we see that
\begin{align}\label{M2}&
\frac{\partial_{z_k}(f)\left(\partial_{z_k}(f)-a\right)\left(\partial_{z_k}(f)-b\right)}{f\left(f-a\right)\left(f-b\right)}
\\&=\frac{\left(\partial_{z_k}(f)\right)^3}{f(f-a)(f-b)}-\frac{(a+b)\left(\partial_{z_k}(f)\right)^2}{f(f-a)(f-b)} +\frac{ab\left(\partial_{z_k}(f)\right)}{f(f-a)(f-b)}\nonumber\\&=
\left(\frac{\partial_{z_k}(f)}{f}\cdot \frac{\partial_{z_k}(f)}{f-a} \cdot \frac{\partial_{z_k}(f)}{f-b}\right)- \left(\frac{a+b}{a-b}\right)\left(\frac{\partial_{z_k}(f)}{f}\right)\bigg[\frac{\partial_{z_k}(f)}{f-a}-\frac{\partial_{z_k}(f)}{f-b}\bigg]\nonumber \\&+ab\bigg[\mathcal{A}_1 \frac{\partial_{z_k}(f)}{f}+\mathcal{A}_2\frac{\partial_{z_k}(f)}{f-a}+\mathcal{A}_3\frac{\partial_{z_k}(f)}{f-b}\bigg],\nonumber
\end{align}  
where $\mathcal{A}_r$ is a suitable constant for $r=1,2,3$. Now from (\ref{M1}) and (\ref{M2}), we get
\begin{align}\label{M3} 
m\left(e^{A_{k1}z_1+A_{k2}z_2+\ldots+A_{kn}z_n+B_k}\right)&=m\left(r, \frac{\partial_{z_k}(f)\left(\partial_{z_k}(f)-a\right)\left(\partial_{z_k}(f)-b\right)}{f\left(f-a\right)\left(f-b\right)}\right) \\
&\leq   \mathcal{C}_1 m\left(r,\frac{\partial_{z_k}(f)}{f}\right)+ \mathcal{C}_2 m\left(r,\frac{\partial_{z_k}(f)}{f-a}\right)+ \mathcal{C}_3 m\left(r,\frac{\partial_{z_k}(f)}{f-b}\right)+O(1),\nonumber
\end{align}
where $\mathcal{C}_1$, $\mathcal{C}_2$ and $\mathcal{C}_3$ are suitable non-zero real numbers. Using Lemma \ref{ln2.8} to (\ref{M3}), we get
\begin{align*}
O(r)=T\left(e^{A_{k1}z_1+A_{k2}z_2+\ldots+A_{kn}z_n+B_k}\right)=m\left(e^{A_{k1}z_1+A_{k2}z_2+\ldots+A_{kn}z_n+B_k}\right)=o(\log r),
\end{align*}
which is impossible.

\smallskip
{\bf Case 2.} Let $abc\neq 0$. We know that 
\begin{align*}
e^{\alpha_k(z)}=C_ke^{A_{k1}z_1+A_{k2}z_2+\ldots+A_{kn}z_n},
\end{align*}
where $(A_{k1},A_{k2},\ldots,A_{kn})\neq (0,0,\ldots,0)$ and $C_k=e^{B_k}$.

Take $\xi_k=(\xi_{k1},\xi_{k2},\ldots,\xi_{kn})\in\mathbb{C}^n\backslash \{0\}$ such that $\xi_{kk}=1$ and $\xi_{kj}=0$ for $j\neq k$. Define a holomorphic mapping $j_{\xi_k}:\mathbb{C}\to \mathbb{C}^n$ by $j_{\xi_k}(z)=z\xi_k$. Clearly $g_k=f_{\xi_k}=f\circ j_{\xi_k}:\mathbb{C}\to \mathbb{C}$ is a holomorphic mapping. Note that $g_k(z)=f(\xi_{k1} z,\xi_{k2} z,\ldots,\xi_{kn}z)$ and if we take $\hat z_k=(\hat z_{k1},\hat z_{k2},\ldots,\hat z_{kn})$, where $\hat z_{kj}=\xi_{kj}z$, $j=1,2,\ldots,n$, then we have
\begin{align}\label{M4}
g^{(1)}_k(z)=\frac{\partial g_k(z)}{\partial z}=\sum\limits_{j=1}^n \xi_{kj}\frac{\partial f(\hat z_k)}{\partial \hat z_{kj}}=\partial_{\hat z_{kk}}(f(\hat z_k)).
\end{align}

Now using (\ref{M4}) to (\ref{MBP1}), we get
\begin{align}\label{M5}
\frac{\left(g_k^{(1)}-a\right)\left(g^{(1)}_k-b\right)\left(g^{(1)}_k-c\right)}{(g_k-a)(g_k-b)(g_k-c)}=C_ke^{A_{kk}z}.
\end{align}

We consider the following two sub-cases.\par

{\bf Sub-case 2.1.} Suppose $A_{kk}=0$. Then from (\ref{M5}), we have
 \begin{align}\label{M6}
\frac{\left(g_k^{(1)}-a\right)\left(g^{(1)}_k-b\right)\left(g^{(1)}_k-c\right)}{(g_k-a)(g_k-b)(g_k-c)}=C_k.
\end{align}

It is easy to conclude from (\ref{M6}) that $E(S,g_k)=E\left(S,g_k^{(1)}\right)$. Now by Theorem B, we deduce that one of the following cases holds
\begin{enumerate}
\item[{(i)}] $f(\tilde z)=c_ke^z$, where $\tilde z=(z_1,\ldots,z_{k-1},z,z_{k+1},\ldots,z_n)$ such that $z_j=0$ for $j\neq k$ and $c_k$ is a non-zero constant in $\mathbb{C}$;
\item[{(ii)}] $f(\tilde z)=c_ke^{-z}+\frac{2}{3}(a+b+c)$ and $(2a-b-c)(2b-c-a)(2c-a-b)=0$, where $\tilde z=(z_1,\ldots, z_{k-1},z, z_{k+1},\ldots,z_n)$ such that $z_j=0$ for $j\neq k$ and $c_k$ is a non-zero constant in $\mathbb{C}$;
\item[{(iii)}] $f(\tilde z)=c_ie^{\frac{-1\pm i\sqrt{3}}{2}z}+\frac{3\pm i\sqrt{3}}{6}(a+b+c)$ and $a^2+b^2+c^2-ab-bc-ca=0$, where $\tilde z=(z_1,\ldots,z_{k-1},z,z_{k+1},\ldots,z_n)$ such that $z_j=0$ for $j\neq k$ and $c_k$ is a non-zero constant in $\mathbb{C}$.
\end{enumerate} 

{\bf Sub-case 2.2.} Suppose $A_{kk}\neq 0$. Then from (\ref{M5}), we get
\begin{align}\label{M7}
\frac{\left(g_k^{(1)}-a\right)\left(g^{(1)}_k-b\right)\left(g^{(1)}_k-c\right)}{(g_k-a)(g_k-b)(g_k-c)}=C_ke^{A_{kk}z}.
\end{align}

On the other hand we know that (see \cite[pp. 286]{Hu-Li-Yang-2003})
\begin{align}\label{M8} 
T_{g_k}(r,0)\leq \frac{1+\theta}{(1-\theta)^{2n-1}} T_f\left(\frac{r}{\theta},0\right),\quad \xi\in\mathbb{C}^n[0,1]\backslash \{0\},
\end{align}
where $0<\theta<1$. Therefore the inequality (\ref{M8}) implies that $\rho(g_k)\leq \rho(f)$ and so $\rho(g_k)\leq 1$. Let 
\begin{align}\label{M9}
h(z)=g_k\left(z/A_{kk}\right).
\end{align}

Obviously $\rho(h)\leq 1$. Now using (\ref{M9}) to (\ref{M7}), we obtain
\begin{align*}
\frac{\left(h^{(1)}-a/A_{kk}\right)\left(h^{(1)}-b/A_{kk}\right)\left(h^{(1)}-c/A_{kk}\right)}{(h-a)(h-b)(h-c)}=C_ke^{z},
\end{align*}

Finally proceeding in the same way as done in the proof of Lemma 6 \cite{Chang-Fang-Zalcman-2007}, we immediately get a contradiction.
\end{proof}

\begin{lem}\label{ln2.9a}\cite{Majumder-Sarkar-2027} Let $g(z)$ be a meromorphic function in $\mathbb{C}^n$. If $\partial^2_{z_iz_i}(g(z))\equiv 0$ for all $i=1,2,\ldots,n$, then $g(z)$ must be a polynomial in $\mathbb{C}^n$.
\end{lem}

\begin{lem}\label{ln2.10} Let $a$, $b$ and $c$ be three distinct finite complex numbers and $A_i$ be non-zero constant in $\mathbb{C}$ for $i=1,2,\ldots,n$. Suppose $f$ is an entire function in $\mathbb{C}^n$ such that $0$ is a Picard exceptional value of $\partial_{z_i}(f)$ for $i=1,2,\ldots,n$ and
\begin{align}\label{M.1}
\frac{(\partial_{z_i}(f)-a)(\partial_{z_i}(f)-b)(\partial_{z_i}(f)-c)}{(f-a)(f-b)(f-c)}=A_i.
\end{align}

Then $f$ must take one of the three forms listed in Theorem \ref{t1.1}.
\end{lem}

\begin{proof} By the first part of the proof of Lemma \ref{ln2.9}, we can prove that $f$ has order at most $1$. Since $f$ is an entire function, using Lemma \ref{ln2.4}, we get
\begin{align*}
T(r,\partial_{z_i}(f)=m(r,\partial_{z_i}(f))\leq m(r,f)+o(T(r,f))=(1+o(1))T(r,f),
\end{align*}
which shows that $\rho(\partial_{z_i}(f))\leq \rho(f)$ and so $\rho(\partial_{z_i}(f))\leq 1$ for $i=1,2,\ldots,n$. Since  $0$ is a Picard exceptional value of $\partial_{z_i}(f)$ for $i=1,2,\ldots,n$, we may take 
\begin{align}\label{M.2}
\partial_{z_i}(f)=C_ie^{a_{i1}z_1+a_{i2}z_2+\ldots+a_{in}z_n},
\end{align}
where $a_{i1},a_{i2},\ldots,a_{in}$ and $C_{i}(\neq 0)$ are constants such that $(a_{i1},a_{i2},\ldots,a_{in})\neq (0,0,\ldots,0)$.
Also from (\ref{M.1}), we have 
\begin{align}\label{M.2a}
&\left(\partial_{z_i}(f(z))\right)^3+L_1\left(\partial_{z_i}(f(z))\right)^2+L_2\left(\partial_{z_i}(f(z))\right)+L_3\\&=
A_i\left(f^3(z)+L_1f^2(z)+L_2f(z)+L_3\right),\nonumber
\end{align}
for $i=1,2,\ldots,n$, where
\begin{align}\label{M.2b}
L_1=-(a+b+c),\quad L_2=ab+bc+ca\quad \text{and}\quad L_3=-abc.
\end{align}

We consider the following two cases.\par

{\bf Case 1.} Suppose $a_{ii}=0$ for $i=1,2,\ldots,n$. Then from (\ref{M.2}), we have $\partial_{z_iz_i}^2(f)\equiv 0$ for $i=1,2,\ldots,n$. Therefore by Lemma \ref{ln2.9a}, we see that $f(z)$ is a polynomial in $\mathbb{C}^n$. Consequently $\partial_{z_i}(f)$ is also a polynomial in $\mathbb{C}^n$ for $i=1,2,\ldots,n$. Since $0$ is a Picard exceptional value of $\partial_{z_i}(f)$ for $i=1,2,\ldots,n$, we get a contradiction.\par

{\bf Case 2.} Suppose $a_{ii}\neq 0$ for atleast one $i\in\{1,2,\ldots,n\}$. For the sake of simplicity, we may assume that $a_{11}\neq 0$. Now taking integration on (\ref{M.2}), we get
\begin{align}\label{M.3}
f(z)=\frac{C_1}{a_{11}}e^{a_{11}z_1+a_{12}z_2+\ldots+a_{1n}z_n}+\alpha_1(z_2,z_3,\ldots,z_n),
\end{align}
where $\alpha_1$ is an entire function in $\mathbb{C}^{n-1}$.
Using Proposition 1.7 \cite{Hu-Li-Yang-2003} to (\ref{M.3}), we deduce that $\rho(\alpha_1)\leq 1$.
Again from (\ref{M.3}), we have 
\begin{align}\label{M.4}
\partial_{z_j}(f(z))=\frac{C_1}{a_{11}}a_{1j}^*e^{a_{11}z_1+a_{12}z_2+\ldots+a_{1n}z_n}+\partial_{z_j}(\alpha_1(z_2,z_3,\ldots,z_n)),
\end{align}
where $j=2,3,\ldots,n$ and we define
\begin{align*}
a^*_{1j}=
\begin{cases}
a_{1j}& \text{if}\; a_{1j}\neq 0\\
0& \text{if}\; a_{1j}=0.
\end{cases}
\end{align*}

We now consider the following two sub-cases.\par

{\bf Sub-case 2.1.} Suppose $\partial_{z_j}(\alpha_1(z_2,z_3,\ldots,z_n))\not\equiv 0$ for atleast one $j\in\{2,3,\ldots,n\}$. We assume that $\partial_{z_k}(\alpha_1(z_2,z_3,\ldots,z_n))\not\equiv 0$, where $k\in\{2,3,\ldots,n\}$. Since $0$ is a Picard exceptional value of $\partial_{z_k}(f(z_1,z_2,\ldots,z_n))$, from (\ref{M.4}), we conclude that $a^*_{1k}=0$ and so 
\begin{align*}
\partial_{z_k}(f(z_1,z_2,\ldots,z_n))\equiv \partial_{z_k}(\alpha_1(z_2.z_3.\ldots,z_n)).
\end{align*}

It is clear that $0$ is also a Picard exceptional value of $\partial_{z_k}(\alpha_1(z_2,z_3,\ldots,z_n))$. Now from (\ref{M.2a}), we get
\begin{align}\label{M.5}
&\left(\partial_{z_k}(\alpha_1(z_2,z_3,\ldots,z_n))\right)^3+L_1\left(\partial_{z_k}(\alpha_1(z_2,z_3,\ldots,z_n))\right)^2+L_2\left(\partial_{z_k}(\alpha_1(z_2,z_3,\ldots,z_n))\right)+L_3\nonumber\\&=
A_k\frac{C_1^3}{a_{11}^3}e^{3(a_{11}z_1+a_{12}z_2+\ldots+a_{1n}z_n)}+A_k\left(3\alpha_1(z_2,z_3,\ldots,z_n)+L_1\right)\frac{C_1^2}{a_{11}^2}e^{2(a_{11}z_1+a_{12}z_2+\ldots+a_{1n}z_n)}\nonumber\\&
+A_k\left(3\alpha^2_1(z_2,z_3,\ldots,z_n)+2L_1\alpha_1(z_2,z_3,\ldots,z_n)+L_2\right)\frac{C_1}{a_{11}}e^{a_{11}z_1+a_{12}z_2+\ldots+a_{1n}z_n}\nonumber\\&
+A_k\left(\alpha_1^3(z_2z_3,\ldots,z_n)+L_1\alpha_1^2(z_2,z_3,\ldots,z_n)+L_2\alpha_1(z_2,z_3,\ldots,z_n)+L_3\right).
\end{align}

Differentiating partially on the both sides of (\ref{M.5}) with respect to $z_1$, we obtain
\begin{align}\label{M.6}
&\frac{3C_1^2}{a_{11}^2}e^{2(a_{11}z_1+a_{12}z_2+\ldots+a_{1n}z_n)}+\left(3\alpha_1(z_2,z_3,\ldots,z_n)+L_1\right)\frac{2C_1}{a_{11}}e^{a_{11}z_1+a_{12}z_2+\ldots+a_{1n}z_n}\nonumber\\&
+\left(3\alpha^2_1(z_2,z_3,\ldots,z_n)+2L_1\alpha_1(z_2,z_3,\ldots,z_n)+L_2\right)=0.
\end{align}

Again differentiating partially with respect to $z_1$, we get from (\ref{M.6}) that
\begin{align}\label{M.7}
&\frac{3C_1}{a_{11}}e^{a_{11}z_1+a_{12}z_2+\ldots+a_{1n}z_n}=-\left(3\alpha_1(z_2,z_3,\ldots,z_n)+L_1\right)
\end{align}

Finally differentiating (\ref{M.7}) partially with respect to $z_1$, we get a contradiction.\par

{\bf Sub-case 2.2.} Suppose $\partial_{z_j}(\alpha_1(z_2,z_3,\ldots,z_n))\equiv 0$ for all $j\in\{2,3,\ldots,n\}$. Consequently it is easy to say that $\alpha_1(z_2,z_3,\ldots,z_n)$ is a constant, say $D$. Then from (\ref{M.3}), we have
\begin{align}\label{M.8}
f(z)=Ce^{a_{11}z_1+a_{12}z_2+\ldots+a_{1n}z_n}+D,
\end{align}
where $C=\frac{C_1}{a_{11}}\neq 0$.
Since $\partial_{z_i}(f)\not\equiv 0$ for $i=1,2,\ldots,n$, from (\ref{M.8}), we can conclude that $a_{1j}\neq 0$ for $j=2,3,\ldots,n$.

Now using (\ref{M.8}) to (\ref{M.2a}), we get
\begin{align}\label{M.9}
&\left(a_{1i}^3-A_i\right)C^3e^{3(a_{11}z_1+a_{12}z_2+\ldots+a_{1n}z_n)}+\left(a_{1i}^2L_1-A_i\left(3D+L_1\right)\right)C^2e^{2(a_{11}z_1+a_{12}z_2+\ldots+a_{1n}z_n)}\\&
+\big({a_{1i}}L_2-A_i\left(3D^2+2L_1D+L_2\right)\big)Ce^{a_{11}z_1+a_{12}z_2+\ldots+a_{1n}z_n}\nonumber\\&+\left(L_3-A_i\left(D^3+L_1D^2+L_2D+L_3\right)\right)=0\nonumber
\end{align}
for $i=1,2,\ldots,n$. One can easily deduce from (\ref{M.9}) that
\begin{align}
	\label{M.10}a_{1i}^3&=A_i,\\ 
\label{M.11} a_{1i}^2L_1&=A_i\left(3D+L_1\right),\\ 
\label{M.12}{a_{1i}}L_2&=A_i\left(3D^2+2L_1D+L_2\right),\\ 
\label{M.13} L_3&=A_i\left(D^3+L_1D^2+L_2D+L_3\right). 
\end{align}
Now we consider the following sub-cases.\par

\smallskip
{\bf Sub-case 2.2.1.} Suppose $L_1=0$. Then from (\ref{M.11}), we have $D=0$. Now we consider the following two sub-cases.\par

\smallskip
{\bf Sub-case 2.2.1.1.} Suppose $L_3\neq 0$. Then from (\ref{M.13}), we get $A_i=1$ for $i=1,2,\ldots,n$. Consequently from (\ref{M.10}), we obtain $a_{1i}^3=1$ for $i=1,2,\ldots,n$, also from (\ref{M.12}) we deduce $L_2=0$. In this case, from (\ref{M.8}), we have $f(z)=Ce^{a_{11}z_1+a_{12}z_2+\ldots+a_{1n}z_n}$, where $a_{11}, a_{12},\ldots,a_{1n}$ and $C$ are non-zero constants such that $a_{1i}^3=1$ for $i=1,2,\ldots,n$.\par

\smallskip
{\bf Sub-case 2.2.1.2.} Suppose $L_3=0$. Note that $a$, $b$ and $c$ are distinct complex numbers. Since $L_3=0$, from (\ref{M.2b}), we deduce that $L_2\neq 0$. Now from (\ref{M.12}), we get $a_{1i}=A_i$ for $i=1,2,\ldots,n$ and so from (\ref{M.10}), we deduce that $a_{1i}^2=1$ for $i=1,2,\ldots,n$. In this case, from (\ref{M.8}), we have $f(z)=Ce^{a_{11}z_1+a_{12}z_2+\ldots+a_{1n}z_n}$, where $a_{11}, a_{12},\ldots,a_{1n}$ and $C$ are non-zero constants such that $a_{1i}^2=1$ for $i=1,2,\ldots,n$.\par

\smallskip
{\bf Sub-case 2.2.2.} Suppose $L_1\neq 0$. Then from (\ref{M.10}) and (\ref{M.11}), we have 
\begin{align}\label{M.14}
	D=\frac{1-a_{1i}}{3a_{1i}}L_1.
\end{align}

Let us discuss the following sub-cases.\par

\smallskip
{\bf Sub-case 2.2.2.1.} Suppose $D=0$. Then from (\ref{M.14}), we deduce that $a_{1i}=1$ for all $i=1,2,\cdots,n$. Now from (\ref{M.8}), we have $f(z)=Ce^{z_1+z_2+\ldots+z_n}$.\par

\smallskip
{\bf Sub-case 2.2.2.2.} Suppose $D\neq 0$. It is easy to conclude from (\ref{M.14}) that $a_{11}=a_{12}=\ldots=a_{1n}$. Clearly from (\ref{M.14}), we have $a_{1i}\neq 1$ for all $i=1,2,\cdots,n$. Now we consider the following two sub-cases.\par

\smallskip
{\bf Sub-case 2.2.2.2.1.} Suppose $a_{1i}=-1$ for all $i=1,2,\cdots,n$. Then from (\ref{M.10}) and (\ref{M.11}), we have $A_i=-1$ for $i=1,2,\ldots,n$ and $D=-\frac{2}{3}L_1$. Here by (\ref{M.13}) we get:
\begin{align*}
	2L_1^3-9L_1L_2+27L_3=0,
\end{align*} 
from which we deduce that $(2a-b-c)(2b-c-a)(2c-a-b)=0$. Thus, from (\ref{M.8}), we see that
\begin{align*}
	f(z)=Ce^{-(z_1+z_2+\ldots+z_n)}+\frac{2}{3}(a+b+c)
\end{align*}
and $(2a-b-c)(2b-c-a)(2c-a-b)=0$.\par

\smallskip
{\bf Sub-case 2.2.2.2.2.} Suppose $a_{1i}\neq -1$ for all $i=1,2,\cdots,n$. Now from  (\ref{M.12}) and (\ref{M.14}), we have
\begin{align}\label{M.15}
	L_2=\frac{L_1^2}{3}.
\end{align}

Clearly from (\ref{M.2b}) and (\ref{M.15}), we deduce that $a^2+b^2+c^2-ab-bc-ca=0$. Again from (\ref{M.10}), (\ref{M.13}), (\ref{M.14}) and (\ref{M.15}), we get
\begin{align}\label{M.16}
	\left(1-a_{1i}^3\right)L_3=\frac{1}{27}\left(1-a_{1i}^3\right)L_1^3.
\end{align}

Now if $1-a_{1i}^3\neq 0$ for $i=1,2,\ldots,n$, then from (\ref{M.16}), we get $L_3=\frac{1}{27}L_1^3$ and so by (\ref{M.15}), we obtain $a=b=c$, which is impossible. Therefore $1-a_{1i}^3=0$ for $i=1,2,\ldots,n$, then $a_{1i}=\frac{-1\pm i\sqrt{3}}{2}$ for all $i=1, 2, \cdots, n$. Consequently, from (\ref{M.8}) and (\ref{M.14}), we obtain
\begin{align*}
	f(z)=Ce^{\left(\frac{-1\pm i\sqrt{3}}{2}\right)\left(z_1+z_2+\cdots+z_n\right)}+\frac{3\pm i\sqrt{3}}{6}(a+b+c)
\end{align*}
and $a^2+b^2+c^2-ab-bc-ca=0$.
\end{proof}

\section{\bf{Proofs of Theorem \ref{t1.1}}}
\begin{proof} Since $E(S,f)=E\left(S,\partial_{z_i}(f)\right)$ for $i=1,2,\ldots,n$, it follows that
\begin{align}\label{Eq1.1}
\frac{(\partial_{z_i}(f)-a)(\partial_{z_i}(f)-b)(\partial_{z_i}(f)-c)}{(f-a)(f-b)(f-c)}=e^{\alpha_i}
\end{align}
where $\alpha_i$ is an entire function in $\mathbb{C}^n$ for $i=1,2,\ldots,n$.

Now we consider the following two cases.\par

\smallskip
{\bf Case 1.} Suppose $\alpha_i$ is non-constant for atleast one $i$, where $i=1,2,\ldots,n$. For the sake of simplicity, we may assume that $\alpha_k$ is non-constant, where $k\in\{1,2,\ldots,n\}$. Then by Lemma \ref{ln2.9}, $f$ must take one of the three forms listed in Lemma \ref{ln2.9}.\par

\smallskip
{\bf Case 2.} Suppose $\alpha_i$ is a constant for $i=1,2,\ldots,n$. Then from (\ref{Eq1.1}), we may assume that
\begin{align}\label{Eq1.2}
(\partial_{z_i}(f)-a)(\partial_{z_i}(f)-b)(\partial_{z_i}(f)-c)=A_i(f-a)(f-b)(f-c).
\end{align}
where $A_i$ is a non-zero constant for $i=1,2,\ldots,n$. By the first part of the proof of Lemma \ref{ln2.9}, we can prove that $f$ has order at most $1$.

In this case, we consider the following two sub-cases.\par

\smallskip
{\bf Sub-case 2.1.} Suppose $0$ is a Picard exceptional value of $\partial_{z_i}(f)$ for $i=1,2,\ldots,n$. Then by Lemma \ref{ln2.10}, $f$ must take one of the three forms listed in Lemma \ref{ln2.10}.\par

\smallskip
{\bf Sub-case 2.2.} Suppose $0$ is not a Picard exceptional value of $\partial_{z_i}(f)$ for atleast one $i$, where $i=1,2,\ldots,n$. Without loss of generality, we may assume that $0$ is not a Picard exceptional value of $\partial_{z_k}(f)$, where $k\in\{1,2,\ldots,n\}$.
Then from (\ref{Eq1.2}), we get
\begin{align}\label{Eq1.3}
\left(\partial_{z_k}(f)\right)^3+L_1\left(\partial_{k}(f)\right)^2+L_2\partial_{k}(f)+L_3=A_k\left(f^3+L_1f^2+L_2f+L_3\right),
\end{align}
where the constants $L_j$ are defined as in (\ref{M.2b}). It is easy to verify from (\ref{Eq1.3}) that $f$ is a transcendental entire function in $\mathbb{C}^n$. Now differentiating partially on the both sides of (\ref{Eq1.3}) with respect to $z_k$, we get
\begin{align}\label{Eq1.4}
\left(3\left(\partial_{z_k}(f)\right)^2+2L_1\left(\partial_{k}(f)\right)+L_2\right)\partial_{z_kz_k}^2(f)=A_k\left(3f^2+2L_1f+L_2\right)\partial_{z_k}(f).
\end{align}

Clearly from (\ref{Eq1.4}), we deduce that $\partial_{z_k}(f)$ is a transcendental entire function in $\mathbb{C}^n$.

\smallskip
Let $f(z)$ and $g(z)$ be two non-constant entire functions in $\mathbb{C}^n$ and $a$ be a finite complex number in $\mathbb{C}$. Let $z_n$ be zeros of $f(z)-a$ with multiplicity $h(n)$. If $z_n$ are also zeros of $g(z)-a$ of multiplicity $h(n)$ at least, then we write $f(z)=a\mapsto g(z)=a$.

\smallskip
We consider the following sub-cases.\par

\smallskip
{\bf Sub-case 2.2.1.} Suppose $\partial_{z_k}(f)=0\mapsto \partial_{z_kz_k}^2(f)=0$. It is clear that $\partial_{z_kz_k}^2(f)/\partial_{z_k}(f)$ is an entire function in $\mathbb{C}^n$. Since $\partial_{z_k}(f)$ has order at most $1$, Lemma \ref{ln2.8} shows that $\partial_{z_kz_k}^2(f)/\partial_{z_k}(f)$ is a non-zero constant $B_k$. On integration, we get
\begin{align*}
\partial_{z_k}(f(z))=D_ke^{B_kz_k+\beta_k(z)},
\end{align*}
where $\beta_k(z)=\beta(z_1,z_2,\ldots,z_{k-1},z_{k+1},\ldots,z_n)$ is an entire function in $\mathbb{C}^{n-1}$. Now in view of Lemma \ref{ln2.6}, we may assume that
\begin{align}\label{Eq1.5}
\partial_{z_k}(f(z))=D_ke^{B_1z_1+B_2z_2+\ldots+B_{k-1}z_{k-1}+B_kz_k+B_{k+1}z_{k+1}+\ldots+B_nz_n},
\end{align}
where $B_1,\ldots,B_{k-1},B_{k+1},\ldots,B_n$ are constants in $\mathbb{C}$. Since $0$ is not a Picard exceptional value of $\partial_{z_k}(f)$, we get a contradiction from (\ref{Eq1.5}).\par

\smallskip
{\bf Sub-case 2.2.2.} Suppose $\partial_{z_k}(f)=0\not\mapsto \partial_{z_kz_k}^2(f)=0$. Let $Z\left(\partial_{z_k}(f)\right)$ be the set of zeros of $\partial_{z_k}(f)$. Then there exists atleast one $z_0\in Z\left(\partial_{z_k}(f)\right)$ such that either $\partial_{z_kz_k}^2(f(z_0))\neq 0$ or $0<\mu^0_{\partial_{z_kz_k}^2(f)}(z_0)<\mu^0_{\partial_{z_k}(f)}(z_0)$. In either case, one can easily deduce from (\ref{Eq1.4}) that $L_2=0$ and so
\begin{align}\label{Eq1.6}
\left(3\partial_{z_k}(f)+2L_1\right)\partial_{z_kz_k}^2(f)=A_k\left(3f+2L_1\right)f.
\end{align}

If possible suppose $L_1=0$. Then from (\ref{Eq1.6}), we have
\begin{align}\label{Eq1.7}
\partial_{z_k}(f)\partial_{z_kz_k}^2(f)=A_kf^2.
\end{align}

Let $z_0$ be a zero of $\partial_{z_k}(f)$ of multiplicity $p_0$. From (\ref{Eq1.7}), it is clear that $z_0$ must be a zero of $f$ of multiplicity $q_0$. Obviously $z_0$ is a zero of $\partial_{z_kz_k}^2(f)$ of multiplicity atleast $p_0-1$. If $z_0$ is a zero of $\partial_{z_kz_k}^2(f)$ of multiplicity exactly $p_0-1$, then from (\ref{Eq1.7}), we deduce that $2p_0-1=2q_0$, which is impossible. Thus, we can conclude that $\partial_{z_k}(f)=0\mapsto \partial_{z_kz_k}^2(f)=0$. Now proceeding in the same way as done in Sub-case 2.2.1, we get a contradiction. 

Hence $L_1\neq 0$. Now we consider the following sub-cases.\par

\smallskip
{\bf Sub-case 2.2.2.1.} Suppose $f=0\mapsto \partial_{z_kz_k}^2(f)=0$. It is clear that $\partial_{z_kz_k}^2(f)/f$ is an entire function in $\mathbb{C}^n$. Since $f$ and $\partial_{z_k}(f)$ has order at most $1$, Lemma \ref{ln2.8} shows that $\partial_{z_kz_k}^2(f)/f$ is a non-zero constant $D_k$. Now substituting $\partial_{z_kz_k}^2(f)=D_kf$ into (\ref{Eq1.6}), we obtain
\begin{align}\label{Eq1.8}
3\partial_{z_k}(f)+2L_1=\tilde A_k (3f+2L_1),
\end{align}
where $\tilde A_k=A_k/D_k$. Differentiating partially on the both sides of (\ref{Eq1.8}) with respect to $z_k$, we get
\begin{align*}
\partial_{z_kz_k}^2(f)=\tilde A_k \partial_{z_k}(f).
\end{align*}

Now proceeding in the same way as done in Sub-case 2.2.1, we get a contradiction.\par 

\smallskip
{\bf Sub-case 2.2.2.2.} Suppose $f=0\not\mapsto \partial_{z_kz_k}^2(f)=0$. Let $Z\left(f\right)$ be the set of zeros of $f$. Then there exists atleast one $z_1\in Z\left(f\right)$ such that $\mu^0_{\partial_{z_kz_k}^2(f)}(z_1)<\mu^0_{f}(z_1)$. In this case, one can easily deduce from (\ref{Eq1.6}) that $\mu^0_{f}(z_1)=1$ and $\mu^0_{\partial_{z_kz_k}^2(f)}(z_1)=0$. As a result from (\ref{Eq1.6}), we have
\begin{align}\label{Eq1.9}
f(z_1)=0\quad \text{and}\quad \partial_{z_k}(f(z_1))=-\frac{2L_1}{3}.
\end{align}

Let $z_2$ be a zero of $3f+2L_1$ of multiplicity $p_2$.

\medskip
First we suppose that $\partial_{z_kz_k}^2(f(z_2))\neq 0$. Thus, from (\ref{Eq1.6}), we obtain 
\begin{align}\label{Eq1.10}
f(z_2)=-\frac{2L_1}{3}\quad \text{and}\quad \partial_{z_k}(f(z_2))=-\frac{2L_1}{3}.
\end{align}

Note that $L_2=0$. Now using (\ref{Eq1.9}) and (\ref{Eq1.10}) to (\ref{Eq1.3}), we get respectively
\begin{align}\label{Eq1.11}
\frac{4}{27}L_1^3+(1-A_k)L_3=0
\end{align}
and
\begin{align}\label{Eq1.12}
(1-A_k)\left(\frac{4}{27}L_1^3+L_3\right)=0.
\end{align}

Then by a simple computation, we conclude from (\ref{Eq1.11}) and (\ref{Eq1.12}) that either $A_k=0$ or $L_1=0$, which contradicts the fact that $A_k\neq 0$ and $L_1\neq 0$.

\medskip
Next we suppose that $\partial_{z_kz_k}^2(f(z_2))=0$. In this case, from (\ref{Eq1.6}), one can easily deduce that $3f+2L_1=0\mapsto \partial_{z_kz_k}^2(f)=0$. It is clear that $\frac{\partial_{z_kz_k}^2(f)}{3f+2L_1}$ is an entire function in $\mathbb{C}^n$. Since $3f+2L_1$ and $\partial_{z_k}(f)$ has order at most $1$, Lemma \ref{ln2.8} shows that $\frac{\partial_{z_kz_k}^2(f)}{3f+2L_1}$ is a non-zero constant $\hat D_k$. Now substituting $\partial_{z_kz_k}^2(f)=\hat D_k(3f+2L_1)$ into (\ref{Eq1.6}), we obtain
\begin{align}\label{Eq1.13}
3\partial_{z_k}(f)+2L_1=\hat B_k f,
\end{align}
where $\hat B_k=A_k/\hat D_k$. Differentiating partially on the both sides of (\ref{Eq1.13}) with respect to $z_k$, we get
\begin{align*}
\partial_{z_kz_k}^2(f)=\frac{\hat B_k}{3} \partial_{z_k}(f).
\end{align*}

Now proceeding in the same way as done in Sub-case 2.2.1, we get a contradiction.

\end{proof}

\medskip

{\bf Statements and declarations:}

\smallskip
\noindent \textbf {Conflict of interest:} The author declares that there are no conflicts of interest regarding the publication of this paper.

\smallskip
\noindent{\bf Funding:} There is no funding received from any organizations for this research work.

\smallskip
\noindent \textbf {Data availability statement:}  Data sharing is not applicable to this article as no database were generated or analyzed during the current study.

\end{document}